\documentclass[11pt, reqno]{amsart}

\usepackage[a4paper,
            top=30mm,bottom=30mm,
            left=35mm,right=35mm]{geometry}

\usepackage{amsmath, amsthm, amsfonts, amssymb, amscd, mathtools, bm}
\usepackage{xcolor}
\usepackage[hidelinks,unicode=true]{hyperref}
\usepackage[shortlabels]{enumitem}
\usepackage[T1]{fontenc}
\usepackage{empheq}
\usepackage{placeins}

\theoremstyle{plain}
\newtheorem{theorem}{Theorem}[section]
\newtheorem{lemma}[theorem]{Lemma}
\newtheorem{corollary}[theorem]{Corollary}
\newtheorem{proposition}[theorem]{Proposition}

\theoremstyle{definition}

\theoremstyle{remark}
\newtheorem{remark}[theorem]{Remark}

\numberwithin{equation}{section}

\newcommand{\dx}{\mathrm{d}x}

\newcommand{\dz}{\mathrm{d}z}
\newcommand{\dt}{\mathrm{d}t}
\newcommand{\dv}{\mathrm{d}v}
\newcommand{\dw}{\mathrm{d}\omega}
\newcommand{\rd}{\mathrm{d}}
\newcommand{\N}{\mathbb{N}}
\newcommand{\R}{\mathbb{R}}

\newcommand{\M}{\mathcal{M}}
\newcommand{\F}{\mathcal{F}}

\title[]{From nonisothermal BGK to Euler Maxwellians \\ via relative entropy}

\author[]{Nuno J. Alves}
\address[N. J. Alves]{
      Applied Mathematics and Computational Sciences (AMCS), Computer, Electrical and Mathematical Sciences and Engineering Division (CEMSE), King Abdullah University of Science and Technology (KAUST), Thuwal, 23955-6900, Kingdom of Saudi Arabia.}
\email{nuno.januarioalves@kaust.edu.sa}

\begin{document}

\begin{abstract}
We study the hydrodynamic limit of the nonisothermal BGK model toward Euler Maxwellians. For a prescribed sufficiently smooth Euler solution, we use relative entropy to compare a BGK solution with the corresponding Euler Maxwellian. The result is conditional: for BGK solutions satisfying a uniform sixth velocity-moment bound and uniform $L^\infty$ bounds on their macroscopic velocity, temperature, and inverse temperature, we obtain a stability estimate uniform in time. The key new ingredient is the control of an additional velocity-cubic term in the relative entropy identity. For well-prepared initial data, this yields strong $L^1$ convergence of the BGK solution and its local Maxwellians to the target Euler Maxwellian, together with convergence of the associated macroscopic quantities.
\end{abstract}

\maketitle
\thispagestyle{empty}

\section{Introduction}

The Bhatnagar--Gross--Krook (BGK) model, introduced in~\cite{bhatnagar1954model},
is a kinetic relaxation equation in which the distribution function
$f^\varepsilon(t,x,v)$ relaxes, on the time scale $\varepsilon>0$, towards a local
Maxwellian determined by its conserved moments. In this paper we consider the
nonisothermal BGK equation on
$[0,T)\times\mathbb R^d_x\times\mathbb R^d_v$,
\begin{equation}\label{eq:BGK}
\partial_t f^\varepsilon + v\cdot\nabla_x f^\varepsilon
= -\frac1\varepsilon\bigl(f^\varepsilon-\mathcal M(f^\varepsilon)\bigr),
\end{equation}
where $\mathcal M(f^\varepsilon)$ is the local Maxwellian with the same mass,
momentum, and kinetic energy as $f^\varepsilon$:
\begin{equation}\label{eq:localM}
\mathcal M(f^\varepsilon)(t,x,v)
=\frac1{(2\pi\theta^\varepsilon(t,x))^{d/2}}\rho^\varepsilon(t,x)
\exp\!\left(-\frac{|v-u^\varepsilon(t,x)|^2}{2\theta^\varepsilon(t,x)}\right).
\end{equation}
The corresponding density, mean velocity, and temperature are given by
\begin{equation}\label{rho_e}
\rho^\varepsilon(t,x)=\int_{\R^d} f^\varepsilon(t,x,v)\,\dv,
\end{equation}
\begin{equation}\label{u_e}
u^\varepsilon(t,x)=
\begin{dcases}
\dfrac1{\rho^\varepsilon(t,x)}\int_{\R^d} v\,f^\varepsilon(t,x,v)\,\dv,
& \text{if } \rho^\varepsilon(t,x)>0,\\
0,& \text{if } \rho^\varepsilon(t,x)=0,
\end{dcases}
\end{equation}
and
\begin{equation}\label{theta_e}
\theta^\varepsilon(t,x)=
\begin{dcases}
\dfrac1{d\,\rho^\varepsilon(t,x)}
\int_{\R^d}|v-u^\varepsilon(t,x)|^2f^\varepsilon(t,x,v)\,\dv,
& \text{if } \rho^\varepsilon(t,x)>0,\\
1,& \text{if } \rho^\varepsilon(t,x)=0.
\end{dcases}
\end{equation}
By construction, $\mathcal M(f^\varepsilon)$ shares with $f^\varepsilon$ the collision
invariants $1$, $v$, and $|v|^2$, so the relaxation term preserves mass, momentum,
and kinetic energy.

Formally, as $\varepsilon\to0$, one expects $f^\varepsilon$ to approach local
Maxwellians and the corresponding moments to satisfy the compressible Euler system with pressure law $p=\rho\theta$,
\begin{equation}\label{eq:Euler}
\partial_t\rho+\nabla_x\cdot(\rho u)=0,
\qquad
\partial_t(\rho u)+\nabla_x\cdot(\rho u\otimes u)+\nabla_x(\rho\theta)=0,
\end{equation}
\begin{equation}\label{eq:EulerE}
\partial_tE+\nabla_x\cdot\bigl(u(E+\rho\theta)\bigr)=0,
\qquad
E=\frac12\rho|u|^2+\frac d2\rho\theta,
\end{equation}
at least before shock formation, that is, on time intervals where $(\rho,u,\theta)$
remains smooth.

The purpose of this paper is to prove a conditional relative entropy stability result for this hydrodynamic limit in the smooth regime. Given a sufficiently smooth Euler solution
$(\rho,u,\theta)$ of~\eqref{eq:Euler}--\eqref{eq:EulerE} on $[0,T)\times\R^d$ with
strictly positive density and temperature, we consider the associated Maxwellian
\begin{equation}\label{eq:targetM}
M(t,x,v)=\frac1{(2\pi\theta(t,x))^{d/2}}\rho(t,x)
\exp\!\left(-\frac{|v-u(t,x)|^2}{2\theta(t,x)}\right).
\end{equation}
For nonnegative phase-space densities $f,g$ of unit mass and with $g>0$, we use the
relative entropy
\begin{equation}\label{def:relent}
H(f\mid g):=\iint_{\R^{2d}} f\log\frac{f}{g}\,\dv\,\dx,
\end{equation}
with the convention that the integrand is zero on the set where $f=0$.

A central point is the computation of $(\partial_t+v\cdot\nabla_x)\log M$. The resulting relative entropy identity contains a traceless stress term and an additional velocity-cubic contribution.
Our argument shows that, for BGK solutions satisfying a uniform sixth velocity-moment bound and uniform $L^\infty$ bounds on their macroscopic velocity, temperature, and inverse temperature, the velocity-cubic term can still be controlled in the smooth regime.
We derive a Gronwall-type estimate for $H(f^\varepsilon\mid M)$ and deduce strong convergence, uniform in time, of $f^\varepsilon$ towards $M$ in $L^1(\R^{2d})$, together with strong convergence of the local Maxwellians $\M(f^\varepsilon)$ towards $M$.
We also obtain the convergence of the associated macroscopic quantities.
Thus the convergence result should be understood as conditional on these a priori bounds. The precise assumptions and the statement of the main theorem are given in Section~\ref{sec:assump}.

The use of relative entropy in stability theory goes back to
Dafermos~\cite{dafermos1979second} and was developed systematically for
fluid-dynamic limits and hyperbolic relaxation by Tzavaras; see, for instance,
\cite{tzavaras2000fluiddynamic,tzavaras2005relative}.
In the kinetic setting, Berthelin and Vasseur~\cite{berthelin2005from} established,
by a relative entropy argument, the convergence from BGK models to multidimensional
Euler equations before shock formation for pressure laws of the form
$p(\rho)=\rho^\gamma$ with $\gamma>1$. The present work follows the same general strategy in the nonisothermal setting,
where the pressure is given by $p(\rho,\theta)=\rho\theta$ and the temperature evolves as an
additional unknown.

Related works by Berthelin and Bouchut
\cite{berthelin2000finite,berthelin2002relaxation,berthelin2002kinetic}
study BGK systems relaxing to gas-dynamic models by means of finite-energy,
kinetic-entropy, and invariant-domain methods.
A discrete-velocity counterpart was investigated by Berthelin, Tzavaras, and
Vasseur~\cite{berthelin2009from}.
More recently, Choi and Hwang~\cite{choi2024weak}, Hwang~\cite{hwang2025classical},
and Koo and Song~\cite{koo2025mild} established weak, classical, and mild solution
theories for related BGK-type models.
See also Choi and Hwang~\cite{choi2024alignment} for a BGK-type hydrodynamic limit
with alignment interactions.
For existence and stability results for the BGK equation itself, we also refer to
Perthame~\cite{perthame1989global}.

The a priori bounds used below are stronger than those supplied directly by the general existence theories cited above. They should be viewed as assumptions defining a class of BGK solutions for which the relative entropy stability estimate applies. Although the works cited above provide useful a priori estimates for related BGK models, these estimates do not by themselves yield the full set of uniform bounds used here. In particular, establishing uniform bounds on $u^\varepsilon$, $\theta^\varepsilon$, and $1/\theta^\varepsilon$ in this nonisothermal setting remains a separate issue, beyond the scope of this paper.

The present work adapts the relative entropy method to the nonisothermal BGK model and identifies the additional heat-flux term that appears in the associated relative entropy identity. To the best of our knowledge, the rigorous BGK hydrodynamic-limit results closest to ours concern regimes in which the pressure depends only on the density, whereas the full nonisothermal BGK limit to Euler Maxwellians has not been treated in this form.

The paper is organized as follows. In Section~\ref{sec:assump} we collect the assumptions and state the main theorem. Section~\ref{sec:convMax} discusses convergence to Maxwellians. In Section~\ref{sec:REI} we
derive the relative entropy identity. Section~\ref{sec:aux} gathers a few auxiliary lemmas used in the proof of the main estimate. Finally, in Sections~\ref{sec:proofThm} and~\ref{sec:proofCor} we prove the theorem and its corollary, respectively.

\section{Assumptions and main theorem} \label{sec:assump}

We now collect the assumptions used in the proof and state the main result. The result is conditional on the existence of a smooth Euler background and a smooth BGK family satisfying the assumptions stated in this section. Throughout the paper, $(\rho,u,\theta)$ denotes a smooth solution of \eqref{eq:Euler}--\eqref{eq:EulerE} on $[0,T)\times\R^d$, and $M$ denotes the associated Maxwellian \eqref{eq:targetM}. For brevity, we write
\[
\|\cdot\|_\infty:=\|\cdot\|_{L^\infty((0,T)\times\R^d)}.
\]

The Euler solution is understood to emanate from initial data
\[
(\rho,u,\theta)|_{t=0}=(\rho_0,u_0,\theta_0).
\]
We assume that these data generate a classical solution on $[0,T)\times\R^d$. The time interval $[0,T)$ is taken inside the lifespan of this smooth solution, that is, before possible loss of regularity or shock formation. Standard local-in-time smooth existence for compressible Euler in nondegenerate regimes follows from the classical theory of quasilinear symmetric hyperbolic systems; see, for instance, \cite{kato1975cauchy,majda1984compressible}. In the present paper we only use the existence of a smooth solution satisfying the bounds in~(B) below.

For each $0<\varepsilon\le1$, let $f^\varepsilon_0\ge0$ be an initial datum of unit mass, and let $f^\varepsilon$ be a smooth nonnegative solution of the BGK equation \eqref{eq:BGK} with
\[
f^\varepsilon(0,x,v)=f^\varepsilon_0(x,v).
\]
The associated macroscopic fields $\rho^\varepsilon$, $u^\varepsilon$, and $\theta^\varepsilon$ are those defined in \eqref{rho_e}--\eqref{theta_e}.
We do not prove existence of such a smooth BGK family in this paper; the main theorem is an a priori relative entropy estimate for any smooth family satisfying the stated assumptions.

The argument relies on the following two groups of hypotheses.

\begin{enumerate}[(A)]
\item \emph{BGK family and uniform bounds.}
The BGK family described above is assumed to have unit total mass,
\[
\iint_{\R^{2d}} f^\varepsilon(t,x,v)\,\dv\,\dx=1
\qquad\text{for all }t\in[0,T),
\]
and to satisfy the following estimates uniformly in $\varepsilon$. Namely, there exist constants
$U>0$, $0<\underline\Theta\le\overline\Theta<\infty$, and $A>0$ such that
\[
\sup_{0<\varepsilon\le1}\|u^\varepsilon\|_\infty\le U,
\]
\[
\sup_{0<\varepsilon\le1}\|\theta^\varepsilon\|_\infty\le \overline\Theta,
\qquad
\sup_{0<\varepsilon\le1}\left\|\frac1{\theta^\varepsilon}\right\|_\infty
\le \frac1{\underline\Theta},
\]
or equivalently,
\[
0<\underline\Theta\le \theta^\varepsilon(t,x)\le \overline\Theta
\quad\text{for all }(t,x)\in[0,T)\times\R^d,
\]
and
\[
\sup_{0<\varepsilon\le1}\sup_{t\in[0,T)}
\iint_{\R^{2d}} |v|^6 f^\varepsilon(t,x,v)\,\dv\,\dx\le A.
\]

\item \emph{Smooth Euler background.}
The triple $(\rho,u,\theta)$ is the smooth solution of \eqref{eq:Euler}--\eqref{eq:EulerE} described above. The Euler solution $(\rho,u,\theta)$ has unit mass,
\[
\int_{\R^d}\rho(t,x)\,\dx=1
\qquad\text{for all }t\in[0,T),
\]
strictly positive density, and temperature bounded away from zero and infinity:
\[
\rho(t,x)>0,
\qquad
0<\underline\theta\le \theta(t,x)\le \overline\theta
\qquad\text{for all }(t,x)\in[0,T)\times\R^d.
\]
Moreover,
\[
u,\ \nabla_xu,\ \nabla_x\log\theta\in L^\infty((0,T)\times\R^d).
\]

\end{enumerate}
\medskip
\par

We also assume sufficient regularity, decay, and integrability so that all relative entropies and entropy productions used in the sequel are finite, and so that all integrations by parts and differentiations under the integral sign are justified.

\begin{theorem}\label{thm_BGK-E}
Let $(\rho,u,\theta)$ be a smooth solution of \eqref{eq:Euler}--\eqref{eq:EulerE} satisfying \emph{(B)}, and let $M$ be the associated Maxwellian \eqref{eq:targetM}. Let $\{f^\varepsilon\}_{0<\varepsilon\le1}$ be a smooth nonnegative BGK family with initial data $f^\varepsilon_0$ as described above, and let $\rho^\varepsilon$, $u^\varepsilon$, and $\theta^\varepsilon$ be defined by \eqref{rho_e}--\eqref{theta_e}. Assume that \emph{(A)} holds and that the regularity, decay, and finite-entropy assumptions stated above are satisfied.

Then there exists a positive constant
\[
C=C\bigl(d,\underline\theta,\overline\theta,U,\underline\Theta,\overline\Theta,
A,\|u\|_\infty,\|\nabla_xu\|_\infty,\|\nabla_x\log\theta\|_\infty\bigr),
\]
independent of $\varepsilon$, such that
\begin{equation}\label{thm_RE_inequality}
\frac{\rd}{\dt} H(f^\varepsilon\mid M)\le C\bigl(H(f^\varepsilon\mid M)+\varepsilon\bigr)
\qquad\text{for all } t\in[0,T).
\end{equation}
Consequently, if at $t=0$,
\begin{equation}\label{thm_conv_hyp}
H(f^\varepsilon_0\mid M_0)\to0,
\qquad\text{as }\varepsilon\to0,
\end{equation}
where $M_0:=M(0,\cdot,\cdot)$, then
\begin{equation}\label{thm_conv_result}
\sup_{t \in [0,T)}H\big(f^\varepsilon(t)\mid M(t)\big)\to0,
\qquad\text{as }\varepsilon\to0.
\end{equation}
\end{theorem}

Theorem~\ref{thm_BGK-E} provides the relative entropy stability estimate and, under well-prepared initial data, yields uniform-in-time convergence in relative entropy. The next corollary records the corresponding kinetic and macroscopic consequences.

\begin{corollary}\label{cor:strongconv}
Under the hypotheses of Theorem~\ref{thm_BGK-E}, together with \eqref{thm_conv_hyp}, one has
\begin{equation} \label{eq:cor_1}
f^\varepsilon\to M,
\qquad
\M(f^\varepsilon)\to M
\qquad
\text{in }L^\infty(0,T;L^1(\R^{2d})),
\end{equation}
as well as
\begin{equation} \label{eq:cor_2}
\rho^\varepsilon\to\rho,
\qquad
\rho^\varepsilon u^\varepsilon\to\rho u,
\qquad
\rho^\varepsilon\theta^\varepsilon\to\rho\theta
\qquad
\text{in }L^\infty(0,T;L^1(\R^d)).
\end{equation}
\end{corollary}

\section{Convergence to Maxwellians} \label{sec:convMax}
In this section we collect a few relative entropy identities involving Maxwellians that will be used repeatedly in the proof of the main theorem. In particular, they provide the link between convergence in relative entropy, strong convergence in phase space, and convergence of the associated macroscopic quantities.

We consider the space
\[
\F=\Bigl\{f\in L^1(\R^{2d}):\; f\ge0,\ \|f\|_{L^1(\R^{2d})}=1,\ \iint_{\R^{2d}}|v|^2f(x,v)\,\dx\,\dv<\infty\Bigr\},
\]
and for each $f\in\F$ we define its density, mean velocity and temperature respectively by
\[
\rho_f(x):=\int_{\R^d} f(x,v)\,\dv,
\]
\[
u_f(x):=
\begin{dcases}
\dfrac{1}{\rho_f(x)}\int_{\R^d} v f(x,v)\,\dv,& \text{if }\rho_f(x)>0,\\
0,& \text{if }\rho_f(x)=0,
\end{dcases}
\]
and
\[
\theta_f(x):=
\begin{dcases}
\dfrac1{d\,\rho_f(x)}\int_{\R^d}|v-u_f(x)|^2 f(x,v)\,\dv,& \text{if }\rho_f(x)>0,\\
1,& \text{if }\rho_f(x)=0.
\end{dcases}
\]
The associated local Maxwellian is
\[
\M(f)(x,v)=\frac{\rho_f(x)}{(2\pi\theta_f(x))^{d/2}}
\exp\!\left(-\frac{|v-u_f(x)|^2}{2\theta_f(x)}\right).
\]
We say that $f$ is a Maxwellian if $f=\M(f)$.

For $f,g\in\F$, the Csisz\'{a}r--Kullback--Pinsker inequality (see~\cite{gilardoni2010pinsker}) states that
\begin{equation}\label{CKP_inequality}
\|f-g\|_{L^1(\R^{2d})}^2\le 2\,H(f\mid g).
\end{equation}

The first lemma gives the basic entropy splitting associated with the Maxwellian projection.
\begin{lemma}\label{lem_pyt}
For $f,g\in\F$, whenever the entropies are finite,
\begin{equation}\label{pyt}
H(g\mid \M(f))=H(g\mid \M(g))+H(\M(g)\mid \M(f)).
\end{equation}
\end{lemma}

\begin{proof}
We write
\[
H(g\mid \M(f))=\iint_{\R^{2d}} g\log\frac{g}{\M(g)}\,\dv\,\dx
+\iint_{\R^{2d}} g\log\frac{\M(g)}{\M(f)}\,\dv\,\dx.
\]
Since $\M(g)$ and $\M(f)$ are Maxwellians, a direct computation shows that
\[
\log\frac{\M(g)}{\M(f)}=a(x)+b(x)\cdot v+c(x)|v|^2
\]
for suitable scalar functions $a,c$ and vector field $b$.
Because $g$ and $\M(g)$ have the same moments against the collision invariants $1$, $v$, and $|v|^2$, we obtain
\begin{align*}
\iint_{\R^{2d}} g\log\frac{\M(g)}{\M(f)}\,\dv\,\dx
& =\iint_{\R^{2d}} \M(g)\log\frac{\M(g)}{\M(f)}\,\dv\,\dx \\
& =H(\M(g)\mid \M(f)).
\end{align*}
This proves \eqref{pyt}.
\end{proof}

We next record the explicit formula for the relative entropy between two Maxwellians. This identity will later allow us to extract quantitative information on the density, velocity, and temperature errors.

\begin{lemma}\label{lem_maxmax}
Let $M_1=M_{\rho_1,u_1,\theta_1}$ and $M_2=M_{\rho_2,u_2,\theta_2}$ be Maxwellians of the form~\eqref{eq:targetM}, with
\[
\rho_1\ge 0,\qquad \rho_2>0,\qquad \theta_1>0,\qquad \theta_2>0.
\]
Then
\begin{equation}\label{eq:maxmax}
\begin{split}
H(M_1\mid M_2)
=&\int_{\R^d} \rho_1\log\frac{\rho_1}{\rho_2}\,\dx
+\frac12\int_{\R^d}\rho_1\frac{|u_1-u_2|^2}{\theta_2}\,\dx \\
&+\frac d2\int_{\R^d}\rho_1\left(\frac{\theta_1}{\theta_2}-1-\log\frac{\theta_1}{\theta_2}\right)\,\dx.
\end{split}
\end{equation}
\end{lemma}

\begin{proof}
We first note that on $\{\rho_1 = 0\}$ we have $M_1(x,v)=0$ for every $v\in\R^d$, and therefore both sides of~\eqref{eq:maxmax} vanish there. It is thus enough to work on $\{\rho_1 > 0\}$.

A direct computation gives
\[
\log\frac{M_1}{M_2}
=\log\frac{\rho_1}{\rho_2}-\frac d2\log\frac{\theta_1}{\theta_2}
-\frac{|v-u_1|^2}{2\theta_1}+\frac{|v-u_2|^2}{2\theta_2}.
\]
Integrating against $M_1$ and using
\[
\int_{\R^d} |v-u_1|^2 M_1\,\dv=d\rho_1\theta_1,
\qquad
\int_{\R^d} |v-u_2|^2 M_1\,\dv=d\rho_1\theta_1+\rho_1|u_1-u_2|^2,
\]
we obtain \eqref{eq:maxmax}.
\end{proof}

As a first consequence of the entropy splitting, convergence in relative entropy towards a Maxwellian implies strong convergence of the associated Maxwellians.

\begin{proposition}\label{prop_conv_1}
Let $\{f_n\}_{n \in \N}\subseteq\F$ be a sequence and let $f\in\F$ satisfy $f=\M(f)$.
If
\[
H(f_n\mid f)\to0\qquad\text{as }n\to\infty,
\]
then
\[
\M(f_n)\to f\qquad\text{in }L^1(\R^{2d})\qquad\text{as }n\to\infty.
\]
\end{proposition}

\begin{proof}
Since $f=\M(f)$, Lemma~\ref{lem_pyt} with $g=f_n$ and $f$ gives
\[
H(f_n\mid f)=H(f_n\mid \M(f_n))+H(\M(f_n)\mid f).
\]
In particular,
\[
H(\M(f_n)\mid f)\le H(f_n\mid f)\to0\qquad\text{as }n\to\infty.
\]
The conclusion follows from the Csisz\'{a}r--Kullback--Pinsker inequality \eqref{CKP_inequality}.
\end{proof}

For completeness, we also record a converse mechanism: if the associated Maxwellians converge in relative entropy and the phase-space densities become asymptotically Maxwellian, then the full distributions converge strongly in phase space.

\begin{proposition}\label{prop_conv_2}
Let $\{f_n\}_{n \in \N}\subseteq\F$ be a sequence and let $f\in\F$ satisfy $f=\M(f)$.
Assume that
\begin{enumerate}[(i)]
\item $H(\M(f_n)\mid \M(f))\to0$ as $n\to\infty$,
\item $H(f_n\mid \M(f_n))\to0$ as $n\to\infty$.
\end{enumerate}
Then
\[
f_n\to f\qquad\text{in }L^1(\R^{2d})\quad\text{as }n\to\infty.
\]
\end{proposition}

\begin{proof}
By Lemma~\ref{lem_pyt} and the fact that $f=\M(f)$,
\[
H(f_n\mid f)=H(f_n\mid \M(f_n))+H(\M(f_n)\mid \M(f)).
\]
The right-hand side converges to zero by assumptions (i)--(ii), hence $H(f_n\mid f)\to0$.
The Csisz\'{a}r--Kullback--Pinsker inequality \eqref{CKP_inequality} then implies
$f_n\to f$ in $L^1(\R^{2d})$.
\end{proof}

\section{Relative entropy identity} \label{sec:REI}

In this section we derive the relative entropy identity that lies at the core of the proof of Theorem~\ref{thm_BGK-E}. The main step is to compute $(\partial_t+v\cdot\nabla_x)\log M$, where $M$ is the target Maxwellian. Once this computation is available, the relative entropy identity follows by a standard differentiation argument.

\begin{lemma}\label{lem_DlogM}
Let $(\rho,u,\theta)$ be a smooth solution of~\eqref{eq:Euler}--\eqref{eq:EulerE}, let $M$ be given by~\eqref{eq:targetM}, and set $c:=v-u(t,x)$.
Then
\begin{equation}\label{eq:DlogM}
\begin{split}
(\partial_t+v\cdot\nabla_x)\log M
= \ & \frac1\theta\left(c\otimes c-\frac{|c|^2}{d}I_d\right):\nabla_xu \\
& +\frac12\,c\cdot\nabla_x\log\theta\left(\frac{|c|^2}{\theta}-(d+2)\right).
\end{split}
\end{equation}
\end{lemma}

\begin{proof}
Set $D_u:=\partial_t+u\cdot\nabla_x$.
The Euler system implies
\[
D_u\log\rho=-\nabla_x\cdot u,
\qquad
D_u u=-\nabla_x\theta-\theta\nabla_x\log\rho,
\]
and
\[
D_u\log\theta=-\frac2d\nabla_x\cdot u.
\]
Since $\partial_t+v\cdot\nabla_x=D_u+c\cdot\nabla_x$ and
\[
\log M=\log\rho-\frac d2\log(2\pi\theta)-\frac{|c|^2}{2\theta},
\]
we compute
\begin{align*}
(\partial_t+v\cdot\nabla_x)\log\rho
&=-\nabla_x\cdot u+c\cdot\nabla_x\log\rho,\\
-\frac d2(\partial_t+v\cdot\nabla_x)\log\theta
&=\nabla_x\cdot u-\frac d2\,c\cdot\nabla_x\log\theta.
\end{align*}
Moreover,
\begin{align*}
(\partial_t+v\cdot\nabla_x)c &=-(\partial_t+v\cdot\nabla_x)u \\&=\nabla_x\theta+\theta\nabla_x\log\rho-(c\cdot\nabla_x)u,
\end{align*}
so that
\begin{align*}
-(\partial_t+v\cdot\nabla_x)\frac{|c|^2}{2\theta}
= \ & -\frac1\theta\,c\cdot\bigl(\nabla_x\theta+\theta\nabla_x\log\rho-(c\cdot\nabla_x)u\bigr) \\
& +\frac{|c|^2}{2\theta}(\partial_t+v\cdot\nabla_x)\log\theta \\
= \ & -c\cdot\nabla_x\log\theta-c\cdot\nabla_x\log\rho
+\frac1\theta c\otimes c: \nabla_xu \\
& -\frac{|c|^2}{d\theta}\nabla_x\cdot u
+\frac{|c|^2}{2\theta}c\cdot\nabla_x\log\theta.
\end{align*}
Adding the three contributions, the terms involving $c\cdot\nabla_x\log\rho$ cancel and we obtain
\begin{align*}
(\partial_t+v\cdot\nabla_x)\log M
= \ & \frac1\theta c\otimes c: \nabla_xu-\frac{|c|^2}{d\theta}\nabla_x\cdot u \\
& +\frac12 c\cdot\nabla_x\log\theta\left(\frac{|c|^2}{\theta}-(d+2)\right),
\end{align*}
which is exactly \eqref{eq:DlogM}.
\end{proof}

We now insert the formula from Lemma~\ref{lem_DlogM} into the time derivative of the relative entropy and obtain the key identity for $H(f^\varepsilon\mid M)$. The following identity is understood under the regularity, decay, integrability, and finiteness assumptions stated in Section~\ref{sec:assump}.

\begin{proposition}\label{prop_REidentity}
Let $f^\varepsilon$ be a smooth solution of \eqref{eq:BGK}, let $(\rho,u,\theta)$ be a smooth solution of \eqref{eq:Euler}--\eqref{eq:EulerE} with associated Maxwellian $M$ given by \eqref{eq:targetM}, and set $c:=v-u(t,x)$. Then
\begin{equation}\label{REIdentity}
\begin{split}
\frac{\rd}{\dt}H(f^\varepsilon\mid M)
= &-\frac1\varepsilon\iint_{\R^{2d}}\bigl(f^\varepsilon-\M(f^\varepsilon)\bigr)\log\frac{f^\varepsilon}{\M(f^\varepsilon)}\,\dv\,\dx \\
&-\int_{\R^d}\frac1\theta\left(\int_{\R^d}\left(c\otimes c-\frac{|c|^2}{d}I_d\right)f^\varepsilon\,\dv\right):\nabla_xu\,\dx \\
&-\frac12\int_{\R^d}\nabla_x\log\theta\cdot\left(\int_{\R^d} c\left(\frac{|c|^2}{\theta}-(d+2)\right)f^\varepsilon\,\dv\right)\dx.
\end{split}
\end{equation}
\end{proposition}

\begin{proof}
We differentiate $H(f^\varepsilon\mid M)$ in time and obtain
\begin{equation}\label{re1}
\frac{\rd}{\rd t}H(f^\varepsilon\mid M)
=\iint_{\R^{2d}} \partial_t f^\varepsilon\,\log\frac{f^\varepsilon}{M}\,\dv\,\dx
+\iint_{\R^{2d}} f^\varepsilon\,\partial_t\log\frac{f^\varepsilon}{M}\,\dv\,\dx.
\end{equation}
Using the BGK equation in the first term, we find
\begin{equation}\label{re11}
\begin{split}
\iint_{\R^{2d}} \partial_t f^\varepsilon\,\log\frac{f^\varepsilon}{M}\,\dv\,\dx
=&-\frac1\varepsilon\iint_{\R^{2d}}\bigl(f^\varepsilon-\M(f^\varepsilon)\bigr)\log\frac{f^\varepsilon}{M}\,\dv\,\dx \\
&-\iint_{\R^{2d}} (v\cdot\nabla_xf^\varepsilon)\log\frac{f^\varepsilon}{M}\,\dv\,\dx.
\end{split}
\end{equation}
Now
\[
\log\frac{f^\varepsilon}{M}
=\log\frac{f^\varepsilon}{\M(f^\varepsilon)}+\log\frac{\M(f^\varepsilon)}{M},
\]
and, as in the proof of Lemma~\ref{lem_pyt}, \[
\log\frac{\M(f^\varepsilon)}{M}=a_0(t,x)+b_0(t,x)\cdot v+c_0(t,x)|v|^2
\]
for suitable scalar functions $a_0,c_0$ and vector field $b_0$.
Since $f^\varepsilon$ and $\M(f^\varepsilon)$ have the same moments against these collision invariants, we get
\[
\iint_{\R^{2d}}\bigl(f^\varepsilon-\M(f^\varepsilon)\bigr)\log\frac{\M(f^\varepsilon)}{M}\,\dv\,\dx=0.
\]
Integrating by parts in $x$ in the transport term gives
\begin{equation}\label{re11b}
\begin{split}
\iint_{\R^{2d}} \partial_t f^\varepsilon\,\log\frac{f^\varepsilon}{M}\,\dv\,\dx
=&-\frac1\varepsilon\iint_{\R^{2d}}\bigl(f^\varepsilon-\M(f^\varepsilon)\bigr)\log\frac{f^\varepsilon}{\M(f^\varepsilon)}\,\dv\,\dx \\
&-\iint_{\R^{2d}} v f^\varepsilon\cdot\nabla_x\log M\,\dv\,\dx.
\end{split}
\end{equation}
For the second term in \eqref{re1}, conservation of total mass yields
\begin{equation}\label{re12}
\begin{split}
\iint_{\R^{2d}} f^\varepsilon\,\partial_t\log\frac{f^\varepsilon}{M}\,\dv\,\dx
&=\iint_{\R^{2d}} f^\varepsilon\,\partial_t(\log f^\varepsilon-\log M)\,\dv\,\dx \\
&=\frac{\rd}{\rd t}\iint_{\R^{2d}} f^\varepsilon\,\dv\,\dx
-\iint_{\R^{2d}} f^\varepsilon\,\partial_t\log M\,\dv\,\dx \\
&=-\iint_{\R^{2d}} f^\varepsilon\,\partial_t\log M\,\dv\,\dx.
\end{split}
\end{equation}
Combining \eqref{re11b} and \eqref{re12} gives
\begin{equation}\label{re2}
\begin{split}
\frac{\rd}{\rd t}H(f^\varepsilon\mid M)
=&-\frac1\varepsilon\iint_{\R^{2d}}\bigl(f^\varepsilon-\M(f^\varepsilon)\bigr)\log\frac{f^\varepsilon}{\M(f^\varepsilon)}\,\dv\,\dx \\
&-\iint_{\R^{2d}} f^\varepsilon(\partial_t+v\cdot\nabla_x)\log M\,\dv\,\dx.
\end{split}
\end{equation}
The identity \eqref{REIdentity} now follows by substituting \eqref{eq:DlogM} from Lemma~\ref{lem_DlogM}.
\end{proof}

\section{Auxiliary lemmas} \label{sec:aux}

In this section we record three auxiliary lemmas used in the proof of Theorem~\ref{thm_BGK-E}. They concern, respectively, the control of a weighted $L^2$-distance by the symmetrized relative entropy, the computation of certain Maxwellian moments, and a quadratic estimate for $r-1-\log r$ on bounded intervals.

\begin{lemma}\label{lem_tri}
For $a,b>0$,
\begin{equation}\label{eq:tri_point}
\frac{2(a-b)^2}{a+b}\le(a-b)\log\frac{a}{b}.
\end{equation}
Consequently, for nonnegative $f,g$ with $f+g>0$ a.e.,
\begin{equation}\label{eq:tri}
\begin{split}
\iint_{\R^{2d}}\frac{(f-g)^2}{f+g}\,\dv\,\dx
& \le\frac12\iint_{\R^{2d}}(f-g)\log\frac{f}{g}\,\dv\,\dx \\
& =\frac12\bigl(H(f\mid g)+H(g\mid f)\bigr).
\end{split}
\end{equation}
\end{lemma}

\begin{proof}
Since
\[
\log\frac{a}{b}=(a-b)\int_0^1 \frac{1}{\tau a+(1-\tau)b}\,\rd\tau,
\]
we have
\[
(a-b)\log\frac{a}{b}=(a-b)^2\int_0^1 \frac{1}{\tau a+(1-\tau)b}\,\rd\tau.
\]
The function $r\mapsto1/r$ is convex on $(0,\infty)$, hence Jensen's inequality yields
\[
\int_0^1 \frac{1}{\tau a+(1-\tau)b}\,\rd\tau
\ge \frac{1}{\int_0^1 (\tau a+(1-\tau)b)\,\rd\tau}=\frac{2}{a+b}.
\]
This proves~\eqref{eq:tri_point}. Integrating the pointwise inequality with $a=f$ and $b=g$ gives~\eqref{eq:tri}.
\end{proof}

\begin{lemma}\label{lem_maxmom}
Let $G=\M(f^\varepsilon)=M_{\rho^\varepsilon,u^\varepsilon,\theta^\varepsilon}$ and let $M=M_{\rho,u,\theta}$ be the target Maxwellian.
Set $c := v - u(t,x)$ and $\delta u:=u^\varepsilon-u$.
Then
\begin{equation}\label{eq:maxmom1}
\int_{\R^d}\left(c\otimes c-\frac{|c|^2}{d}I_d\right)G\,\dv
=\rho^\varepsilon\left(\delta u\otimes\delta u-\frac{|\delta u|^2}{d}I_d\right),
\end{equation}
and
\begin{equation}\label{eq:maxmom2}
\int_{\R^d}c\left(\frac{|c|^2}{\theta}-(d+2)\right)G\,\dv
=\rho^\varepsilon\,\delta u\left((d+2)\left(\frac{\theta^\varepsilon}{\theta}-1\right)+\frac{|\delta u|^2}{\theta}\right).
\end{equation}
\end{lemma}

\begin{proof} Write $c=(v-u^\varepsilon)+(u^\varepsilon-u)=:\omega+\delta u$. Under the Maxwellian $G$ we have 
\[ \int_{\R^d}\omega\,G\,\dv=0, \qquad \int_{\R^d}\omega\otimes\omega\,G\,\dv=\rho^\varepsilon\theta^\varepsilon I_d, \qquad \int_{\R^d}|\omega|^2\,G\,\dv= d \rho^\varepsilon\theta^\varepsilon. \]
Expanding $c\otimes c$ and $|c|^2$ immediately gives~\eqref{eq:maxmom1}. 

For~\eqref{eq:maxmom2}, we write \[ c\left(\frac{|c|^2}{\theta}-(d+2)\right) =(\omega+\delta u)\left(\frac{|\omega|^2+2\omega\cdot\delta u+|\delta u|^2}{\theta}-(d+2)\right). \] 
After integrating against $G$ we are left with \begin{align*} 
\int_{\R^d}c\left(\frac{|c|^2}{\theta}-(d+2)\right)G\,\dv &=\frac{\delta u}{\theta}\int_{\R^d}|\omega|^2 G\,\dv +\frac2\theta\int_{\R^d}\omega(\omega\cdot\delta u)G\,\dv \\ &\qquad+\frac{|\delta u|^2}{\theta}\,\delta u\int_{\R^d}G\,\dv-(d+2)\delta u\int_{\R^d}G\,\dv. 
\end{align*} 
Using 
\[\int_{\R^d}\omega(\omega\cdot\delta u)G\,\dv=\rho^\varepsilon\theta^\varepsilon\delta u, \qquad \int_{\R^d}G\,\dv=\rho^\varepsilon, \] 
we obtain \eqref{eq:maxmom2}.
 \end{proof}

\begin{lemma}\label{lem_tempquad}
Fix $\overline r>0$. Then there exists $C=C(\overline r)>0$ such that
\[
(r-1)^2\le C\bigl(r-1-\log r\bigr)
\qquad\text{for all }r\in(0,\overline r].
\]
\end{lemma}

\begin{proof}
The function $\psi(r):=r-1-\log r$ is $C^2$ and strictly convex on $(0,\infty)$,
with $\psi(1)=\psi'(1)=0$ and $\psi''(r)=1/r$.
By Taylor's theorem about the point $r=1$, for every $r>0$ there exists
\[
\xi\in (\min\{1,r\},\max\{1,r\})
\]
such that
\[
\psi(r)=\frac12\,\psi''(\xi)(r-1)^2.
\]
If $r\in(0,\overline r]$, then $\xi\le \max\{1,\overline r\}$, hence
\[
\psi(r)\ge \frac{1}{2\max\{1,\overline r\}}(r-1)^2.
\]
This proves the claim.
\end{proof}

\section{Proof of Theorem~\ref{thm_BGK-E}} \label{sec:proofThm}
Throughout this proof, $C$ denotes a positive constant, possibly changing from line to line, depending only on the quantities appearing in the statement of Theorem~\ref{thm_BGK-E}.

Set $c:=v - u$, $G:=\M(f^\varepsilon)$ and
\begin{equation}\label{eq:De}
D_\varepsilon(t):=\iint_{\R^{2d}}\bigl(f^\varepsilon-G\bigr)\log\frac{f^\varepsilon}{G}\,\dv\,\dx
=H(f^\varepsilon\mid G)+H(G\mid f^\varepsilon)\ge0.
\end{equation}
By the relative entropy identity in Proposition~\ref{prop_REidentity},
\begin{equation}\label{eq:REid2}
\frac{\rd}{\dt}H(f^\varepsilon\mid M)
=-\frac1\varepsilon D_\varepsilon(t)+\mathcal R_u(t)+\mathcal R_\theta(t),
\end{equation}
where
\begin{align*}
\mathcal R_u(t)
&:=-\int_{\R^d}\frac1\theta\left(\int_{\R^d}\left(c\otimes c-\frac{|c|^2}{d}I_d\right)f^\varepsilon\,\dv\right):\nabla_xu\,\dx,\\
\mathcal R_\theta(t)
&:=-\frac12\int_{\R^d}\nabla_x\log\theta\cdot\left(\int_{\R^d}c\left(\frac{|c|^2}{\theta}-(d+2)\right)f^\varepsilon\,\dv\right)\dx.
\end{align*}
We split each remainder into a Maxwellian part and an extra term,
\[
\mathcal R_u=\mathcal R_u^{\mathrm M}+\mathcal R_u^{\mathrm{E}},
\qquad
\mathcal R_\theta=\mathcal R_\theta^{\mathrm M}+\mathcal R_\theta^{\mathrm{E}},
\]
where the Maxwellian part is obtained by replacing $f^\varepsilon$ with $G=\M(f^\varepsilon)$, and the extra term is obtained by replacing $f^\varepsilon$ with $f^\varepsilon-G$.

\medskip
\noindent\textbf{Step 1: Maxwellian parts.}
Using Lemma~\ref{lem_maxmom}, we obtain
\[
\mathcal R_u^{\mathrm M}
=-\int_{\R^d}\frac{\rho^\varepsilon}{\theta}
\left(\delta u\otimes\delta u-\frac{|\delta u|^2}{d}I_d\right):\nabla_xu\,\dx,
\qquad
\delta u:=u^\varepsilon-u.
\]
Since $\theta\ge\underline\theta$ and
$\bigl|\delta u\otimes\delta u-(|\delta u|^2/d)I_d\bigr|\le2|\delta u|^2$,
we infer
\begin{equation}\label{eq:RuMbd}
|\mathcal R_u^{\mathrm M}|
\le \frac{2}{\underline\theta}\|\nabla_xu\|_\infty\int_{\R^d}\rho^\varepsilon|\delta u|^2\,\dx.
\end{equation}
On the other hand, Lemma~\ref{lem_maxmax} gives
\begin{align*}
H(G\mid M)
= \ & \int_{\R^d}\rho^\varepsilon\log\frac{\rho^\varepsilon}{\rho}\,\dx 
+\frac12\int_{\R^d}\rho^\varepsilon\frac{|\delta u|^2}{\theta}\,\dx \\
& +\frac d2\int_{\R^d}\rho^\varepsilon\left(\frac{\theta^\varepsilon}{\theta}-1-\log\frac{\theta^\varepsilon}{\theta}\right)\dx,
\end{align*}
whence
\[
\int_{\R^d}\rho^\varepsilon|\delta u|^2\,\dx
\le 2\overline\theta\,H(G\mid M).
\]
Therefore,
\begin{equation}\label{eq:RuMbd2}
|\mathcal R_u^{\mathrm M}|\le C\,H(G\mid M),
\end{equation}
where $C$ depends only on $\underline\theta$, $\overline\theta$, and $\|\nabla_xu\|_\infty$.

Now, set $r:=\theta^\varepsilon/\theta$. By Lemma~\ref{lem_maxmom},
\[
\mathcal R_\theta^{\mathrm M}
=
-\frac12\int_{\R^d}\rho^\varepsilon\,\delta u\cdot\nabla_x\log\theta
\left((d+2)(r-1)+\frac{|\delta u|^2}{\theta}\right)\dx.
\]
Hence
\begin{align*}
|\mathcal R_\theta^{\mathrm M}|
&\le \frac12\|\nabla_x\log\theta\|_\infty
\int_{\R^d}\rho^\varepsilon |\delta u|
\left((d+2)|r-1|+\frac{|\delta u|^2}{\theta}\right)\dx \\
&= \frac12\|\nabla_x\log\theta\|_\infty
\int_{\R^d}\rho^\varepsilon
\left((d+2)|\delta u|\,|r-1|+\frac{|\delta u|^3}{\theta}\right)\dx.
\end{align*}
Moreover,
\[
|\delta u|\le \|u^\varepsilon\|_\infty+\|u\|_\infty\le U+\|u\|_\infty.
\]
Using $|ab|\le\frac12(a^2+b^2)$ and $\theta\ge\underline\theta$, we obtain
\[
|\delta u|\,|r-1|
\le \frac12\bigl(|\delta u|^2+(r-1)^2\bigr),
\]
and
\[
\frac{|\delta u|^3}{\theta}
\le \frac{1}{\underline\theta}|\delta u|^3
\le \frac{U+\|u\|_\infty}{\underline\theta}\,|\delta u|^2.
\]
Substituting these bounds into the previous estimate yields
\begin{align*}
|\mathcal R_\theta^{\mathrm M}|
&\le \frac12\|\nabla_x\log\theta\|_\infty
\int_{\R^d}\rho^\varepsilon
\left(\frac{d+2}{2}\bigl(|\delta u|^2+(r-1)^2\bigr)
+\frac{U+\|u\|_\infty}{\underline\theta}|\delta u|^2\right)\dx \\
&\le C\|\nabla_x\log\theta\|_\infty
\int_{\R^d}\rho^\varepsilon\bigl(|\delta u|^2+(r-1)^2\bigr)\dx,
\end{align*}
where $C$ depends on $d$, $\underline\theta$, $U$, and $\|u\|_\infty$.

Since $0<\theta^\varepsilon\le \overline\Theta$ and $\theta\ge \underline\theta$, the ratio $r=\theta^\varepsilon/\theta$
satisfies $0<r\le \frac{\overline\Theta}{\underline\theta}$. Applying Lemma~\ref{lem_tempquad} with $\overline r=\overline\Theta/\underline\theta$, we obtain
\[
(r-1)^2\le C\left(r-1-\log r\right),
\]
where $C=C(\overline\Theta/\underline\theta)$.
Invoking again Lemma~\ref{lem_maxmax}, we conclude that
\begin{equation}\label{eq:RthetaMbd}
|\mathcal R_\theta^{\mathrm M}|\le C\,H(G\mid M),
\end{equation}
where now $C$ depends on
$d$, $\underline\theta$, $\overline\theta$, $U$, $\overline\Theta$,
$\|u\|_\infty$, and $\|\nabla_x\log\theta\|_\infty$.

Combining \eqref{eq:RuMbd2} and \eqref{eq:RthetaMbd} gives
\begin{equation}\label{eq:Rmacro}
|\mathcal R_u^{\mathrm M}|+|\mathcal R_\theta^{\mathrm M}|\le C\,H(G\mid M).
\end{equation}

\medskip
\noindent\textbf{Step 2: Extra terms.}
Set $h:=f^\varepsilon-G$. By definition of the extra terms, we have
\[
\mathcal R_u^{\mathrm E}=-\iint_{\R^{2d}}\Phi_u\,h\,\dv\,\dx,
\qquad
\mathcal R_\theta^{\mathrm E}=-\iint_{\R^{2d}}\Phi_\theta\,h\,\dv\,\dx,
\]
where
\[
\Phi_u:=\frac1\theta\left(c\otimes c-\frac{|c|^2}{d}I_d\right):\nabla_xu,
\qquad
\Phi_\theta:=\frac12\,\nabla_x\log\theta\cdot c\left(\frac{|c|^2}{\theta}-(d+2)\right).
\]

For any measurable $\Phi$, we may write
\[
\iint_{\R^{2d}}\Phi h\,\dv\,\dx
=
\iint_{\R^{2d}}
\Bigl(\Phi\sqrt{f^\varepsilon+G}\Bigr)
\Bigl(\frac{h}{\sqrt{f^\varepsilon+G}}\Bigr)\,\dv\,\dx,
\]
and hence, by Cauchy--Schwarz,
\begin{equation}\label{eq:CSw}
\left|\iint_{\R^{2d}}\Phi h\,\dv\,\dx\right|
\le
\left(\iint_{\R^{2d}}\Phi^2(f^\varepsilon+G)\,\dv\,\dx\right)^{1/2}
\left(\iint_{\R^{2d}}\frac{h^2}{f^\varepsilon+G}\,\dv\,\dx\right)^{1/2},
\end{equation}
with the convention $h^2/(f^\varepsilon+G)=0$ on $\{f^\varepsilon+G=0\}$. 

By Lemma~\ref{lem_tri} and \eqref{eq:De},
\begin{equation}\label{eq:tri_use}
\iint_{\R^{2d}}\frac{h^2}{f^\varepsilon+G}\,\dv\,\dx
\le \frac12\,D_\varepsilon(t).
\end{equation}

We next estimate the growth of $\Phi_u$ and $\Phi_\theta$ in $v$.
Since
\[
|c|=|v-u|\le |v|+\|u\|_\infty,
\]
and since $0<\underline\theta\le \theta\le \overline\theta$, we have
\begin{align*}
|\Phi_u|
&\le \frac{\|\nabla_xu\|_\infty}{\underline\theta}
\left|c\otimes c-\frac{|c|^2}{d}I_d\right| \\
&\le C|c|^2 \\
&\le C(1+|v|^2).
\end{align*}
Similarly,
\begin{align*}
|\Phi_\theta|
&\le \frac12\|\nabla_x\log\theta\|_\infty\,|c|
\left(\frac{|c|^2}{\underline\theta}+d+2\right) \\
&\le C\bigl(|c|+|c|^3\bigr) \\
& \le C(1+|v|^3),
\end{align*}
where in the last step we used the inequality $|v| \leq 1 + |v|^3$.

Therefore,
\begin{equation}\label{eq:Phi_growth}
\Phi_u^2+\Phi_\theta^2\le C\bigl(1+|v|^6\bigr),
\end{equation}
where $C$ depends only on
$d$, $\underline\theta$, $\|u\|_\infty$, $\|\nabla_xu\|_\infty$, and $\|\nabla_x\log\theta\|_\infty$.

Using \eqref{eq:Phi_growth}, the unit mass of $f^\varepsilon$, and the assumed sixth
moment bound, we get
\[
\iint_{\R^{2d}}(\Phi_u^2+\Phi_\theta^2)f^\varepsilon\,\dv\,\dx
\le
C\iint_{\R^{2d}}(1+|v|^6)f^\varepsilon\,\dv\,\dx
\le C.
\]

We now estimate the same quantity with $G$. Writing $\omega:=v-u^\varepsilon$,
we have
\[
G
=
\frac{\rho^\varepsilon}{(2\pi\theta^\varepsilon)^{d/2}}
\exp\!\left(\!-\frac{|\omega|^2}{2\theta^\varepsilon}\right).
\]
Hence
\begin{align*}
\int_{\R^d}|v|^6 G\,\dv
&=
\frac{\rho^\varepsilon}{(2\pi\theta^\varepsilon)^{d/2}}
\int_{\R^d}|\omega+u^\varepsilon|^6
\exp\!\left(-\frac{|\omega|^2}{2\theta^\varepsilon}\right)\dw \\
&\le
C\,\frac{\rho^\varepsilon}{(2\pi\theta^\varepsilon)^{d/2}}
\int_{\R^d}\bigl(|\omega|^6+|u^\varepsilon|^6\bigr)
\exp\!\left(-\frac{|\omega|^2}{2\theta^\varepsilon}\right)\dw,
\end{align*}
where we used the elementary inequality $|a+b|^6\le C\bigl(|a|^6+|b|^6\bigr)$.
Since $|u^\varepsilon|\le U$ and $\theta^\varepsilon\le\overline\Theta$, a change of
variables $\omega=\sqrt{\theta^\varepsilon}z$ shows that
\[
\frac{1}{(2\pi\theta^\varepsilon)^{d/2}}
\int_{\R^d}|\omega|^6
\exp\!\left(-\frac{|\omega|^2}{2\theta^\varepsilon}\right)\dw
=
(\theta^\varepsilon)^3(2\pi)^{-d/2}\int_{\R^d}|z|^6e^{-|z|^2/2}\, \dz
\le C,
\]
and also
\[
\frac{1}{(2\pi\theta^\varepsilon)^{d/2}}
\int_{\R^d}
\exp\!\left(-\frac{|\omega|^2}{2\theta^\varepsilon}\right)\dw
=1.
\]
Therefore
\[
\int_{\R^d}|v|^6 G\,\dv\le C\,\rho^\varepsilon,
\]
with $C=C(d,U,\overline\Theta)$.
Using \eqref{eq:Phi_growth} again and $\int_{\R^d}\rho^\varepsilon\,\dx=1$, we conclude that
\[
\iint_{\R^{2d}}(\Phi_u^2+\Phi_\theta^2)G\,\dv\,\dx\le C.
\]

Combining the previous two bounds, we obtain
\[
\iint_{\R^{2d}}(\Phi_u^2+\Phi_\theta^2)(f^\varepsilon+G)\,\dv\,\dx\le C,
\]
for some constant $C$ independent of $\varepsilon$ and $t$.
Applying \eqref{eq:CSw} with $\Phi=\Phi_u$ and $\Phi=\Phi_\theta$, and using
\eqref{eq:tri_use}, we infer that
\begin{align*}
|\mathcal R_u^{\mathrm E}|+|\mathcal R_\theta^{\mathrm E}|
&\le
\left(\iint_{\R^{2d}}\Phi_u^2(f^\varepsilon+G)\,\dv\,\dx\right)^{1/2}
\left(\frac12D_\varepsilon(t)\right)^{1/2} \\
&\qquad+
\left(\iint_{\R^{2d}}\Phi_\theta^2(f^\varepsilon+G)\,\dv\,\dx\right)^{1/2}
\left(\frac12D_\varepsilon(t)\right)^{1/2} \\
&\le C\sqrt{D_\varepsilon(t)}.
\end{align*}
Finally, Young's inequality gives
\begin{equation}\label{eq:Rmic2}
|\mathcal R_u^{\mathrm E}|+|\mathcal R_\theta^{\mathrm E}|
\le \frac{1}{2\varepsilon}D_\varepsilon(t)+\frac\varepsilon2 C.
\end{equation}

\medskip
\noindent\textbf{Step 3: Gronwall estimate.}
Combining \eqref{eq:REid2}, \eqref{eq:Rmacro}, and \eqref{eq:Rmic2}, we obtain
\[
\frac{\rd}{\dt}H(f^\varepsilon\mid M)
\le -\frac1\varepsilon D_\varepsilon(t)+\frac{1}{2\varepsilon}D_\varepsilon(t)+C\,H(G\mid M)+C\varepsilon.
\]
By Lemma~\ref{lem_pyt},
\[
H(f^\varepsilon\mid M)=H(f^\varepsilon\mid G)+H(G\mid M),
\]
so that $H(G\mid M)\le H(f^\varepsilon\mid M)$.
Therefore,
\begin{equation}\label{eq:RE_aux}
\frac{\rd}{\dt}H(f^\varepsilon\mid M)\le C\,H(f^\varepsilon\mid M)+C\varepsilon,
\end{equation}
establishing~\eqref{thm_RE_inequality}.
Gronwall's lemma then yields
\begin{equation}\label{eq:stability}
\sup_{t\in[0,T)}H\big(f^\varepsilon(t)\mid M(t)\big)
\le e^{CT}\bigl(H(f^\varepsilon_0\mid M_0)+\varepsilon\bigr).
\end{equation}
By the well-preparedness assumption~\eqref{thm_conv_hyp} it follows from~\eqref{eq:stability} that
\begin{equation*}
\sup_{t\in[0,T)}H\big(f^\varepsilon(t)\mid M(t)\big)\to0
\qquad\text{as }\varepsilon\to0,
\end{equation*}
which finishes the proof. \qed

\begin{remark}
Since $M_0$ is a Maxwellian, Lemma~\ref{lem_pyt} gives
\[
H(f^\varepsilon_0\mid M_0)
=
H\big(f^\varepsilon_0\mid \M(f^\varepsilon_0)\big)
+
H\big(\M(f^\varepsilon_0)\mid M_0\big).
\]
In particular, the well-preparedness condition
\[
H(f^\varepsilon_0\mid M_0)\to0
\qquad\text{as }\varepsilon\to0
\]
is equivalent to
\[
H\big(f^\varepsilon_0\mid \M(f^\varepsilon_0)\big)\to0,
\qquad
H\big(\M(f^\varepsilon_0)\mid M_0\big)\to0,
\qquad\text{as }\varepsilon\to0.
\]
\end{remark}

\section{Proof of Corollary~\ref{cor:strongconv}} \label{sec:proofCor}

We first prove the convergence in phase space. Since both $f^\varepsilon(t)$ and $M(t)$
have unit mass for every $t\in[0,T)$, the Csisz\'ar--Kullback--Pinsker inequality gives
\[
\|f^\varepsilon(t)-M(t)\|_{L^1(\R^{2d})}^2
\le 2\,H\bigl(f^\varepsilon(t)\mid M(t)\bigr).
\]
Taking the supremum over $t\in[0,T)$ and using~\eqref{thm_conv_result}, we obtain
\[
f^\varepsilon\to M
\qquad\text{in }L^\infty(0,T;L^1(\R^{2d})).
\]

Moreover, since $M(t)$ is a Maxwellian for every $t\in[0,T)$, Lemma~\ref{lem_pyt}
yields
\[
H\bigl(f^\varepsilon(t)\mid M(t)\bigr)
=
H\bigl(f^\varepsilon(t)\mid \M(f^\varepsilon)(t)\bigr)
+
H\bigl(\M(f^\varepsilon)(t)\mid M(t)\bigr).
\]
Hence
\[
H\bigl(\M(f^\varepsilon)(t)\mid M(t)\bigr)
\le H\bigl(f^\varepsilon(t)\mid M(t)\bigr)
\qquad\text{for all }t\in[0,T).
\]
Since $\M(f^\varepsilon)(t)$ and $M(t)$ also have unit mass, the
Csisz\'ar--Kullback--Pinsker inequality gives
\[
\|\M(f^\varepsilon)(t)-M(t)\|_{L^1(\R^{2d})}^2
\le 2\,H\bigl(\M(f^\varepsilon)(t)\mid M(t)\bigr)
\le 2\,H\bigl(f^\varepsilon(t)\mid M(t)\bigr).
\]
Taking the supremum in time and using~\eqref{thm_conv_result}, we conclude that
\[
\M(f^\varepsilon)\to M
\qquad\text{in }L^\infty(0,T;L^1(\R^{2d})).
\]

We now turn to the macroscopic quantities. Since both $\M(f^\varepsilon)$ and $M$
are Maxwellians of the form \eqref{eq:targetM}, Lemma~\ref{lem_maxmax} gives
\begin{equation}\label{eq:maxmax_eps}
\begin{split}
H(\M(f^\varepsilon)\mid M)
=&\int_{\R^d}\rho^\varepsilon\log\frac{\rho^\varepsilon}{\rho}\,\dx
+\frac12\int_{\R^d}\rho^\varepsilon\frac{|u^\varepsilon-u|^2}{\theta}\,\dx \\
&+\frac d2\int_{\R^d}\rho^\varepsilon
\left(\frac{\theta^\varepsilon}{\theta}-1-\log\frac{\theta^\varepsilon}{\theta}\right)\dx.
\end{split}
\end{equation}
Since all terms on the right-hand side of \eqref{eq:maxmax_eps} are nonnegative, while
\[
H\bigl(\M(f^\varepsilon)(t)\mid M(t)\bigr)\le H\bigl(f^\varepsilon(t)\mid M(t)\bigr)
\qquad\text{for all }t\in[0,T),
\]
it follows from \eqref{thm_conv_result} that
\begin{equation}\label{eq:rhoent_to0}
\sup_{t\in[0,T)}\int_{\R^d}\rho^\varepsilon(t)\log\frac{\rho^\varepsilon(t)}{\rho(t)}\,\dx\to0,
\end{equation}
\begin{equation}\label{eq:vel_to0}
\sup_{t\in[0,T)}\int_{\R^d}\rho^\varepsilon(t)\frac{|u^\varepsilon(t)-u(t)|^2}{\theta(t)}\,\dx\to0,
\end{equation}
and
\begin{equation}\label{eq:tempent_to0}
\sup_{t\in[0,T)}\int_{\R^d}\rho^\varepsilon(t)
\left(\frac{\theta^\varepsilon(t)}{\theta(t)}-1-\log\frac{\theta^\varepsilon(t)}{\theta(t)}\right)\dx\to0.
\end{equation}

Applying the Csisz\'ar--Kullback--Pinsker inequality on $\R^d$ to the probability
densities $\rho^\varepsilon(t)$ and $\rho(t)$, we deduce from \eqref{eq:rhoent_to0} that
\[
\rho^\varepsilon\to\rho
\qquad\text{in }L^\infty(0,T;L^1(\R^d)).
\]

Next,
\begin{align*}
\|\rho^\varepsilon u^\varepsilon-\rho u\|_{L^1(\R^d)}
\le \ &  \|\rho^\varepsilon(u^\varepsilon-u)\|_{L^1(\R^d)}
   +\|u\|_\infty\|\rho^\varepsilon-\rho\|_{L^1(\R^d)} \\
\le \ &  \left(\int_{\R^d}\rho^\varepsilon\,\dx\right)^{1/2}
     \left(\int_{\R^d}\rho^\varepsilon|u^\varepsilon-u|^2\,\dx\right)^{1/2} \\
   & +\|u\|_\infty\|\rho^\varepsilon-\rho\|_{L^1(\R^d)} \\
\le \ & \overline\theta^{1/2}
\left(\int_{\R^d}\rho^\varepsilon\frac{|u^\varepsilon-u|^2}{\theta}\,\dx\right)^{1/2} \\
&+\|u\|_\infty\|\rho^\varepsilon-\rho\|_{L^1(\R^d)}.
\end{align*}
Since $\int_{\R^d}\rho^\varepsilon\,\dx=1$ and both terms on the right-hand side
converge to zero uniformly in time by~\eqref{eq:vel_to0} and the convergence of
$\rho^\varepsilon$, we obtain
\[
\rho^\varepsilon u^\varepsilon\to\rho u
\qquad\text{in }L^\infty(0,T;L^1(\R^d)).
\]

Finally, since $0<\theta^\varepsilon\le\overline\Theta$ and $\theta\ge\underline\theta$,
we have
\[
0<\frac{\theta^\varepsilon}{\theta}\le \frac{\overline\Theta}{\underline\theta}.
\]
Applying Lemma~\ref{lem_tempquad} with
\[
r=\frac{\theta^\varepsilon}{\theta},
\]
we infer that there exists a constant $C=C(\overline\Theta/\underline\theta)>0$ such that
\[
\left(\frac{\theta^\varepsilon}{\theta}-1\right)^2
\le
C\left(\frac{\theta^\varepsilon}{\theta}-1-\log\frac{\theta^\varepsilon}{\theta}\right).
\]
Since $\theta\le\overline\theta$, multiplying by $\rho^\varepsilon\theta^2$ and integrating yields
\[
\int_{\R^d}\rho^\varepsilon|\theta^\varepsilon-\theta|^2\,\dx
\le
C\,\overline\theta^2
\int_{\R^d}\rho^\varepsilon
\left(\frac{\theta^\varepsilon}{\theta}-1-\log\frac{\theta^\varepsilon}{\theta}\right)\dx.
\]
Hence~\eqref{eq:tempent_to0} implies
\[
\sup_{t\in[0,T)}\int_{\R^d}\rho^\varepsilon(t)|\theta^\varepsilon(t)-\theta(t)|^2\,\dx\to0.
\]
Therefore,
\begin{align*}
\|\rho^\varepsilon\theta^\varepsilon-\rho\theta\|_{L^1(\R^d)}
\le \ &  \|\rho^\varepsilon(\theta^\varepsilon-\theta)\|_{L^1(\R^d)}
   +\|\theta(\rho^\varepsilon-\rho)\|_{L^1(\R^d)} \\
\le \ & \left(\int_{\R^d}\rho^\varepsilon\,\dx\right)^{1/2}
     \left(\int_{\R^d}\rho^\varepsilon|\theta^\varepsilon-\theta|^2\,\dx\right)^{1/2} \\
   & +\|\theta\|_\infty\|\rho^\varepsilon-\rho\|_{L^1(\R^d)}.
\end{align*}
Since $\int_{\R^d}\rho^\varepsilon\,\dx=1$, both terms on the right-hand side converge
to zero uniformly in time, so we conclude that
\[
\rho^\varepsilon\theta^\varepsilon\to\rho\theta
\qquad\text{in }L^\infty(0,T;L^1(\R^d)).
\]
This completes the proof. \qed

\section*{Acknowledgments}
The author thanks Professor A.~E.~Tzavaras for introducing him to this topic and the anonymous referee for suggestions that improved the exposition. This publication is based upon work supported by King Abdullah University of Science and Technology (KAUST) under Award No. ORFS-CRG12-2024-6430.

\end{document}